\documentclass[12pt,twoside]{amsart}
\usepackage{graphicx}
\usepackage{color}
\usepackage{amssymb}
\usepackage{amsmath}
\usepackage{amsthm}
\usepackage{amsfonts}
\usepackage{amsmath, amsfonts, amssymb}
\newtheorem{theorem}{Theorem}[section]
\newtheorem{corollary}[theorem]{Corollary}

\newtheorem{proposition}[theorem]{Proposition}

\newtheorem{definition}[theorem]{Definition}
\numberwithin{equation}{section}
\begin{document} 
\small \begin{center}{\textit{ In the name of
			Allah, the Beneficent, the Merciful.}}\end{center}
\large

	\title{ON FINITE DIMENSIONAL ALGEBRAS WITH ONLY TRIVIAL DERIVATIONS(AUTOMORPHISMS) and SIMPLE ALGEBRAS} %% Article title
		\author{U.Bekbaev} %% Author name
	\thanks{{ uralbekbaev@gmail.com}}
	\begin{center} \address{Turin Polytechnic University in Tashkent, Tashkent, Uzbekistan}
	\end{center}	
	
	%% Abstract
	\begin{abstract}
		This paper deals with $n$-dimensional algebras, over any field, which have only trivial derivation (automorphism) and simple algebras. It is shown that the corresponding sets of algebras are not empty and, in algebraically closed field case, they are dense subsets of the variety of $n$-dimensional algebras with respect to the Zariski topology. Moreover, an inductive construction method is offered to create these kind of algebras as well.  In two-dimensional case a complete classifications, up to isomorphism, of such algebras are provided.\end{abstract}
	
	\maketitle
		
		

\section{Introduction}

Description of automorphisms and derivations of a given algebra is an essential part of the investigation of that algebra. Therefore there are many researches dealing with this problem. The topic of simple algebras is another point of interest for many researches as well.  

Consider the variety of $n$-dimensional algebras $\mathcal{A}=\mathcal{A}_n$ over any field $\mathbb{F}$. How many elements of $\mathcal{A}$ do have only trivial derivation(automorphism), how many  elements of $\mathcal{A}$ are simple?
This paper deals with these questions. It is shown that the corresponding sets of algebras are not empty and they are dense subsets of $\mathcal{A}$, in the Zariski topology, in the case of algebraically closed field $\mathbb{F}$. Moreover, for the each case an inductive method of construction of such algebras is provided and it is shown that the set of all algebras isomorphic to the constructed ones is a dense subset of $\mathcal{A}$ whenever the basic field $\mathbb{F}$ is algebraically closed.

Investigation the  relations among these three subsets is an interesting thing. In this paper some of such relations and applications of the main results are presented as well.

In two-dimensional case more accurate results can be provided. They are presented, without proofs, in the Appendix section. In that section         
a complete classification, up to isomorphism, of all elements $\mathcal{A}_2$, which have only trivial derivation(trivial automorphism, are simple), is given.
By the use of given classification results one can find out relations among the tree subsets in question. 

A motivation to us, for the looking at these problems in the presented form,  were the classification of all two-dimensional algebras, description of their derivations(automorphisms) and local, 2-local derivations(automorphisms)  \cite{ahmed1,ahmed2,bekbaev23}. 
%\ref{subsec1}.

The paper is  organized in the following way. In section 3, which comes after the Introduction and Preliminaries sections, we deal with algebras which have only trivial derivation. Sections 4 and 5 deal with algebras which have only trivial automorphism and are simple, respectively. In section 6 we consider some applications and in section 7 we provide, in two-dimensional case, final answers to some problems considered in the previous sections.

\section{Preliminaries} Let $\mathbb{F}$ be any field. By an algebra 
$\mathbb{A}$ over $\mathbb{F}$ we mean vector space $\mathbb{A}$ with a fixed bilinear map $\cdot :\mathbb{A}\times\mathbb{A}\rightarrow \mathbb{A}$ and use $x\cdot y=xy$ for the value of this bilinear map at $(x,y)\in \mathbb{A}\times\mathbb{A}$. Let $Lin(\mathbb{A})$ stand for the set all linear maps from $\mathbb{A}$ to $\mathbb{A}$, as a vector space over $\mathbb{F}$.

\begin{definition} An element $g\in Lin(\mathbb{A})$ ($D\in Lin(\mathbb{A})$) is said to be an automorphism (respectively,  a derivation) if $g$ is invertible and for any $x,y\in \mathcal{A}$ the equality $g(xy)=g(x)g(y)$  (respectively, $D(xy)=D(x)y+xD(y)$) holds true.\end{definition}

We use notation $Der(\mathbb{A})$ ($Aut(\mathbb{A})$) for the set of all derivations(respectively, automorphisms) of the algebra $\mathbb{A}$.

Further we consider only $n$ dimensional algebras and $I_n=I$ stands for the identity matrix, where $n$ is a fixed natural number. We use $B\otimes C$ for the tensor product of matrices $B=(b_{ij})$ and $C$, where $B\otimes C=(b_{ij}C)$.

If $\mathbb{A}$ is an algebra and $e=(e_1,e_2,...,e_n)$ is a linear basis for $\mathbb{A}$ over $\mathbb{F}$, as a vector space, then  one can attach  to  $\mathbb{A}$ a $n \times n^2$ size matrix $\mathbb{A}_e$ (called the matrix of structure constants of $\mathbb{A}$ with respect to the basis $e=(e_1,e_2,...,e_n)$, shortly MSC) $$A_e=\left(\begin{array}{ccccccccccccc}a_{11}^1&a_{12}^1&...&a_{1n}^1&a_{21}^1&a_{22}^1&...&a_{2n}^1&...&a_{n1}^1&a_{n2}^1&...&a_{nn}^1\\ a_{11}^2&a_{12}^2&...&a_{1n}^2&a_{21}^2&a_{22}^2&...&a_{2n}^2&...&a_{n1}^2&a_{n2}^2&...&a_{nn}^2 \\
	...&...&...&...&...&...&...&...&...&...&...&...&...\\ a_{11}^n&a_{12}^n&...&a_{1n}^n&a_{21}^n&a_{22}^n&...&a_{2n}^n&...&a_{n1}^n&a_{n2}^n&...&a_{nn}^n\end{array}\right)$$ as follows
$$e_i \cdot e_j=\sum\limits_{k=1}^n a_{ij}^ke_k, \ \mbox{where}\ i,j=1,2,...,n.$$
Note that under a change of the basis $f=eg^{-1}$ the MSC $A_e$ is changed according to the rule $$A_f=gA_e(g^{-1})^{\otimes 2}.$$
In this way the variety of $n$-dimensional algebras $\mathcal{A}$ becomes $Mat(n\times n^2,\mathbb{F}).$

Further it is assumed that a basis $e$ is fixed, we don't make difference between an algebra $\mathbb{A}$ and its MSC $A_e=A$, as well as between a derivation(an automorphism) and its matrix, with respect to this basis, elements of $\mathbb{A}$ in this basis will be presented as elements of $\mathbb{F}^n$ in column form. 

% if $u^1,u^2,...,u^n\in F^n$ then $[u^1|u^2|...|u^n]$ stands for the matrix whose columns are  $u^1,u^2,...,u^n$. We use $|X|$ for the cardinality of a set $X$.
We use notations $A=(A_1|A_2|...|A_n]$, where $e_ie=eA_i$ at $i=1,2,...,n$, that is,	
$$A_1=\left(\begin{array}{cccc}a_{11}^1&a_{12}^1&...&a_{1n}^1\\
	a_{11}^2&a_{12}^2&...&a_{1n}^2\\
	...&...&...&...\\ a_{11}^n&a_{12}^n&...&a_{1n}^n\end{array}\right),\
%A_2=\left(\begin{array}%%%{cccc}a_{21}^1&a_{22}^1&...&a_{2n}^1\\
	%	a_{21}^2&a_{22}^2&...&a_{2n}^2\\
	%	...&...&...&...\\ a_{21}^n&a_{22}^n&...&a_{2n}^n\end{array}\right)
..., A_n=\left(\begin{array}{cccc}a_{n1}^1&a_{n2}^1&...&a_{nn}^1\\
	a_{n1}^2&a_{n2}^2&...&a_{nn}^2\\
	...&...&...&...\\ a_{n1}^n&a_{n2}^n&...&a_{nn}^n\end{array}\right),$$
$tr(A_i)$ means the ordinary trace of $A_i$.

Further we use the following fact that $D\in Der(\mathbb{A})$ ($g\in Aut(\mathbb{A})$)  if and only if $DA=A(D\otimes I+I\otimes D)$ 
(respectively, $gA=A(g\otimes g)$ and $\det(g)\neq 0$), which can be easily derived from the identity 
$$\mathbf{u}\mathbf{v}=eA(u\otimes v),$$ where
$\mathbf{u}=eu, \mathbf{v}=ev$, $u,v\in\mathbb{F}^n$  .

\section{Derivations}
The following result is about algebras which have only trivial derivation. Let $\mathcal{A}_{n, triv.der}$ stand for the set of all such algebras, that is $\mathcal{A}_{n, triv.der}=\{\mathbb{A}\in\mathcal{A} :\ Der(\mathbb{A})=\{0\}\}$
\begin{theorem} The following statements on $\mathcal{A}_{n, triv.der}$  are true\\
	1. $\mathcal{A}_{n, triv.der}\neq \emptyset$.\\
	2. If $\mathbb{F}$ is algebraically closed then $\mathcal{A}_{n, triv.der}$ is an open subset of $\mathcal{A}$, in Zariski topology, and therefore it is dense in  $\mathcal{A}$.\end{theorem}

\begin{proof} Let $A$ be the MSC of an algebra $\mathbb{A}$, $X\in Lin(\mathbb{A})$ be the matrix of a derivation from $Der(\mathbb{A})$, that is the equality $XA=A(X\otimes I+I\otimes X)$ is true. It is known that this equality can be represented in the form
	$M(A)vec(X)=0$, where $M(A)$ is $n^3\times n^2$ size matrix depending on $A$, $vec(X)$ is $X$ written in a column matrix form.
	
	The set $\{A: rank(M(A))=n^2\}$ represent all $n$-dimensional algebras $\mathbb{A}$ for which $Der(\mathbb{A})=\{0\}$ and it is an open subset of the variety of $n$-dimensional algebras in algebraically closed field $\mathbb{F}$ case. Therefore to complete the proof of the theorem it is enough to show that $\{A: rank(M(A))=n^2\}$ is not empty. For that consider the following (commutative, associative) algebra $\mathbb{A}$ defined by 
	\[e_i^2=e_i\ \mbox{whenever}\  i=1,2,...,n,\ e_ie_j=0\ \mbox{whenever}\  i\neq j.\]
	
	If $D\in Der(\mathbb{A})$ and $D(e_i)=\sum_{j=1}^nd_{ji}e_j$ then, due to $D(e_i)= D(e_i^2)=2D(e_i)e_i$, one has $ \sum_{j=1}^nd_{ji}e_j=2\sum_{j=1}^nd_{ji}e_je_i=2d_{ii}e_i$. Therefore $d_{ii}=2d_{ii}$ and $d_{ij}=0$ whenever $i\neq j$, that is $D=0$.\end{proof}

The following theorem is about a construction of $n$-dimensional algebra, which has only trivial derivation, by the use of a $n-1$- dimensional algebra which has the same property.  

\begin{theorem} Let $A'=(A'_2\vert A'_3\vert...\vert A'_n)\in Mat((n-1)\times (n-1)^2,\mathbb{F})$, $A'_i\in Mat((n-1)\times (n-1),\mathbb{F})$, where $i=2,3,...n$, be any matrices  with the following properties\\ a) if $D'A'=A'(D'\otimes I_{n-1}+I_{n-1}\otimes D')$ for some $D'\in Mat((n-1)\times (n-1),\mathbb{F})$ then $D'=0$,\\
	b) rank of the block-column matrix $A'_{bl}$ with blocks $A'_i+tr(A'_i)I_{n-1}$, where $i=2,3,...,n$, is $n-1$.
	
	In this case $A=(A_1\vert A_2\vert ...\vert A_n)\in Mat(n\times n^2,\mathbb{F})$, where $A_1\in Mat(n\times n,\mathbb{F})$, $tr(A_1)\neq 0$, $A_i=\begin{pmatrix}a^1_{i1}&0\\ \overline{a}_{i1}&A'_i\end{pmatrix}$, $a^1_{i1}=-tr(A'_i)$ and $\overline{a}_{i1}\in \mathbb{F}^{n-1}$ are any column vectors for all $i\in \{2,3,...,n\}$, has property: if $DA=A(D\otimes I+I\otimes D)$ for some $D\in Mat(n\times n,\mathbb{F})$ then $D=0$. \end{theorem}

\begin{proof} Let $(d_{ij})=D\in Mat(n,\mathbb{F})$ and $DA=A(D\otimes I+I\otimes D)$. The last equality is equivalent to  $DA_i=A_iD+\sum_{j=1}^nA_jd_{ji}$, where $i=1,2,...,n$. Therefore one has  
	$tr(DA_i)=tr(A_iD)+\sum_{j=1}^ntr(A_j)d_{ji}=tr(A_iD)+tr(A_1)d_{1i}$, that is $d_{1i}=0$, whenever $i=1,2,...,n$, so
	$D=\begin{pmatrix}0&0\\ \mathbf{d}&D'\end{pmatrix}$. Moreover $DA_i=A_iD+\sum_{j=1}^nA_jd_{ji}$, means that 
	\[\begin{pmatrix}0&0\\ \mathbf{d}&D'\end{pmatrix}\begin{pmatrix}a^1_{i1}&0\\ \overline{a}_{i1}&A'_i\end{pmatrix}=\begin{pmatrix}a^1_{i1}&0\\ \overline{a}_{i1}&A'_i\end{pmatrix}\begin{pmatrix}0&0\\ \mathbf{d}&D'\end{pmatrix}+\sum_{j=1}^n\begin{pmatrix}a^1_{j1}&0\\ \overline{a}_{j1}&A'_j\end{pmatrix}d_{ji}\] and it can be represented in the following form 
	$$\begin{pmatrix} 0&0\\ a^1_{i1}\mathbf{d}+D'\overline{a}_{i1}&D'A'_i\end{pmatrix}=\begin{pmatrix}0&0\\ A'_i\mathbf{d}&A'_iD'\end{pmatrix}+\sum_{j=2}^n\begin{pmatrix}a^1_{j1}&0\\ \overline{a}_{j1}&A'_j\end{pmatrix}d_{ji}.$$ In particular, if $i>1$ then $D'A'_i=A'_iD'+\sum_{j=2}^nA'_jd'_{ji}$, so $D'=0$. Once again, in particular,   
	$(A'_i-a^1_{i1}I_{n-1})\mathbf{d}=(A'_i+tr(A'_i)I_{n-1})\mathbf{d}=0$ for all $i=2,...,n$, that is, $A'_{bl}\mathbf{d}=0$. So $\mathbf{d}=0$ due to  $rk(A'_{bl})=n-1$. It means that $D=0$. 
\end{proof}

Let $\mathcal{A}_{2,constr.der}$ stand for the set of all two-dimensional algebras given by MSC   $$A=\left(
\begin{array}{cccc}
	\alpha_1 & \alpha_2 &1+\alpha_2 & \alpha_4 \\
	\beta_1 & -\alpha_1 & 1-\alpha_1 & -\alpha_2
\end{array}\right),\ \mbox{where}\ \mathbf{c}=(\alpha_1, \alpha_2, \alpha_4, \beta_1) \in \mathbb{F}^4.$$ All these algebras have only trivial derivation and trivial automorphism, see \cite{eshmirzaev24}.  
By the use of an element $A'\in\mathcal{A}_{2,constr.der}$, for which rank of $A'_{bl}$ is $2$, and method given in Theorem 3.2 one can construct $3$-dimensional algebra $A$ with trivial derivation. Let $\mathcal{A}_{3,constr.der}$ stand for the set of such obtained $3$-dimensional algebras. In similar way one can use $\mathcal{A}_{3,constr.der}$ to construct the set of $4$-dimensional algebras $\mathcal{A}_{4,constr.der}$, and etcetera, one can construct $\mathcal{A}_{n,constr.der}$ elements of which have only trivial derivation.

The next result states that the set of all algebras, isomorphic to at least one the elements of $\mathcal{A}_{n,constr.der}$, is a dense subset of the variety of $n$-dimensional algebras if the basic field  
$\mathbb{F}$ is algebraically closed. 
%If an algebra $A$ has only trivial derivation then any algebra isomorphic to it , that is $g^{-1}Ag^{\otimes 2}$, where $g\in GL(n,\mathbb{F})$, also has the same property.

\begin{proposition}  In the case of algebraically closed $\mathbb{F}$ the set $$\{g^{-1}Ag^{\otimes 2}: A\in\mathcal{A}_{n,constr.der}, g\in GL(n,\mathbb{F})\}$$ is a dense subset of $\mathcal{A}$.
\end{proposition}

\begin{proof} Let $X_i$ stand for the set of variables $(x_{ij}^k)_{j,k=1,2,...n}$. Let $f[X_1\vert X_2\vert...\vert X_n]$ be any polynomial over $\mathbb{F}$ and such that $f[g^{-1}Ag^{\otimes 2}]=0$ whenever $g\in GL(n,\mathbb{F})$ and  $A\in \mathcal{A}_{n,constr.der}$.
	In particular, $f[A]=0$ and due to the fact that for $A_1$ one has only the constraint $tr(A_1)\neq 0$ one can conclude that $f$ doesn't depend on $X_1$, that is $f[X_1\vert X_2\vert...\vert X_n]=f[X_2\vert X_3\vert...\vert X_n]$. But the change of the basis $e=(e_1,e_2,...,e_n$ to 
	$e'=(e_2,e_1,...,e_n$, roughly speaking, changes the old $A_2$ to old $A_1$ and therefore $f[X_2\vert X_3\vert...\vert X_n]$ does not depend on $X_2$ as well, so $f[X_2\vert X_3\vert...\vert X_n]=f[X_3\vert X_4\vert...\vert X_n]$ and et cetera. \end{proof}
\section{Automotphisms}
In this section we deal with algebras which have only trivial automorphism, provide an inductive construction method of making such algebras.
\begin{theorem} Let $A'=(A'_2\vert A'_3\vert...\vert A'_n)\in Mat((n-1)\times (n-1)^2,\mathbb{F})$, $A'_i\in Mat((n-1)\times (n-1),\mathbb{F})$, where $i=2,3,...,n,$ be any matrices  with the following properties\\ a) if $g'A'=A'(g'\otimes g')$ for some $g'\in GL(n-1,\mathbb{F})$ then $g'=I_{n-1}$,\\
	b)  rank of the block-column matrix $A'_{bl}$ with blocks $A'_i+tr(A'_i)I_{n-1}$, where $i=2,3,...,n$, is $n-1$
	
	In this case $A=(A_1\vert A_2\vert ...\vert A_n)\in Mat(n\times n^2,\mathbb{F})$, where $A_1\in Mat(n\times n,\mathbb{F})$, $tr(A_1)\neq 0$, $A_i=\begin{pmatrix}a^1_{i1}&0\\ \overline{a}_{i1}&A'_i\end{pmatrix}$, $a^1_{i1}=-tr(A'_i)$ and $\overline{a}_{i1}\in \mathbb{F}^{n-1}$ are any column vectors for all $i\in \{2,3,...,n\}$, has property: if $gA=A(g\otimes g)$ for some $g\in GL(n,\mathbb{F})$ then $g=I$. \end{theorem}

\begin{proof}Let $(g_{ij})=g\in GL(n,\mathbb{F})$ and $gA=A(g\otimes g)$. The last equality is equivalent to  $A_i=g^{-1}\sum_{j=1}^nA_jg_{ji}g$, where $i=1,2,...,n$. Therefore $tr(A_i)=\sum_{j=1}^ntr(g^{-1}A_jg)g_{ji}=tr(A_1)g_{1i}$, that is, $g_{11}=1$ and at 
	$i>1$ one has $g_{1i}=0$, so 
	$g=\begin{pmatrix}1&0\\ \mathbf{g}_0&g'\end{pmatrix}$. In $i>1$ case the equality $gA_i=\sum_{j=1}^nA_jg_{ji}g$, that is 
	$$\begin{pmatrix}1&0\\ \mathbf{g}_0&g'\end{pmatrix}\begin{pmatrix}a^1_{i1}&0\\ \overline{a}_{i1}&A'_i\end{pmatrix}=\sum_{j=1}^n\begin{pmatrix}a^1_{j1}&0\\ \overline{a}_{j1}&A'_j\end{pmatrix}\begin{pmatrix}1&0\\ \mathbf{g}_0&g'\end{pmatrix}g_{ji},$$ can be represented in the following form $\begin{pmatrix}a^1_{i1}&0\\ a^1_{i1}\mathbf{g}_0+g'\overline{a}_{i1}&g'A'_i\end{pmatrix}=$ $$\sum_{j=1}^n\begin{pmatrix}a^1_{j1}&0\\ \overline{a}_{j1}+A'_j\mathbf{g}_0&A'_jg'\end{pmatrix}g_{ji}=\begin{pmatrix}\sum_{j=2}^na^1_{j1}g'_{ji}&0\\ \sum_{j=2}^n(\overline{a}_{j1}+A'_j\mathbf{g}_0)g'_{ji}&\sum_{j=2}^n(A'_jg')g'_{ji}\end{pmatrix}.$$ In particular, it implies that if $i>1$ then $g'A'_i=\sum_{j=2}^nA'_jg'g'_{ji}$, so $g'=I_{n-1}$. Therefore the above equality is nothing but  
	$\begin{pmatrix}a^1_{i1}&0\\ a^1_{i1}\mathbf{g}_0+\overline{a}_{i1}&A'_i\end{pmatrix}=\begin{pmatrix}a^1_{i1}&0\\ \overline{a}_{i1}+A'_i\mathbf{g}_0&A'_i\end{pmatrix}$ and $(A'_i-a^1_{i1}I_{n-1})\mathbf{g}_0=(A'_i+tr(A'_i)I_{n-1})\mathbf{g}_0=0$ for all $i>1$, that is $A'_{bl}\mathbf{g}_0=0$. So  $\mathbf{g}_0=0$ and  $g=I$ due to $rk(A'_{bl})=n-1$. 
\end{proof}

Let $\mathcal{A}_{2,constr.aut}$ stand  for the set of all two-dimensional algebras given by MSC   $$A=\left(
\begin{array}{cccc}
	\alpha_1 & \alpha_2 &1+\alpha_2 & \alpha_4 \\
	\beta_1 & -\alpha_1 & 1-\alpha_1 & -\alpha_2
\end{array}\right),\ \mbox{where}\ \mathbf{c}=(\alpha_1, \alpha_2, \alpha_4, \beta_1) \in \mathbb{F}^4.$$  As in the derivation case by the use  of  $\mathcal{A}_{2,constr.aut}$ one can construct, following the method presented in Theorem 4.1, a set of three-dimensional algebras $\mathcal{A}_{3,constr.aut}$  which have only trivial automorphism and etcetera, one can construct $\mathcal{A}_{n,constr.aut}$ elements of which have only trivial automorphism.

It is clear that according to the Theorems 3.2 and Theorems 4.1  the equality $$\mathcal{A}_{n,aut,constr}=\mathcal{A}_{n,der,constr}$$ is valid.

%\begin{proposition}  In the case of algebraically closed $\mathbb{F}$ the set $\mathcal{A}_{n,aut.constr}$ is a dense subset of $\mathcal{A}$.
%\end{proposition}
Now we are ready to prove the following result.	
Let $\mathcal{A}_{n, triv.aut}=\{\mathbb{A}\in\mathcal{A} :\ Aut(\mathbb{A})=\{I\}\}$.	
\begin{theorem} The following statements are true:\\ 
	1. $\mathcal{A}_{n, triv.aut}\neq \emptyset$.\\
	2. If $\mathbb{F}$ is algebraically closed field then $\mathcal{A}_{n, triv.der}\cap\mathcal{A}_{n, triv.aut}$ is a dense, in the Zariski topology, subset of $\mathcal{A}$.	\\
	3.  $\mathcal{A}_{n, triv.der}\setminus \mathcal{A}_{n, triv.aut}\neq \emptyset$.\end{theorem}

\begin{proof}  	
	The validity of the first statement is clear due to the constructive approach in Theorem 4.1.
	
	The second statement is true due to the inclusion $$\mathcal{A}_{n,constr.aut}=\mathcal{A}_{n,constr.der}\subset \mathcal{A}_{n, triv.der}\cap\mathcal{A}_{n, triv.aut}$$ and Proposition 3.3 . 
	
	To prove the last statement $\mathcal{A}_{n, triv.der}\setminus \mathcal{A}_{n, triv.autr}\neq \emptyset$ one can consider $n$-dimensional (commutative, associative) algebra, used in Theorem 3.1, defined by  \[e_i^2=e_i\ \mbox{whenever}\  i=1,2,...,n,\ e_ie_j=0\ \mbox{whenever}\  i\neq j,\] and automorphism $g$ defined by, for example, $g(e_1)=e_2, g(e_2)=e_1$ and $g(e_j)=e_j$ in other cases.
\end{proof}

\section{Simple algebras}
In this section we deal with simple algebras. An algebra is said to be a simple algebra if its two-sided ideals are trivial. 

Let $B=(B_1,B_2, ...,B_m)$, where $B_i\in Mat(n\times n,\mathbb{F})$, $i=1,2,...,m$, be a fixed system of matrices. A vector subspace $V$ of $\mathbb{F}^n$ is said to be $B$-invariant if  $B_i(V)\subset V$ for all $i=1,2,...,m$. It is clear that if $\lambda_1,\lambda_2,...,\lambda_m$ are any scalars then $V$ is $B$-invariant if and only if $V$ is $(B_1-\lambda_1I_n,B_2-\lambda_2I_n, ...,B_m-\lambda_mI_n)$- invariant. 

An algebra given by its MSC $A=(A_1\vert A_2\vert ...\vert A_n)$ is a simple algebra if and only if the system of matrices $\overline{A}=(A_1,A_2,...,A_n,A^o_1,A^o_2,...,A^o_n)$ has only trivial invariant subspace, where %$A=(A_1,A_2,...,A_n), A^o=(A^o_1,A^o_2,...,A^o_n)$,
$ee_i=eA^o_i$, that is $A^o_i=\left(\begin{array}{cccc}a_{1i}^1&a_{2i}^1&...&a_{ni}^1\\
	a_{1i}^2&a_{2i}^2&...&a_{ni}^2\\
	...&...&...&...\\ a_{1i}^n&a_{2i}^n&...&a_{ni}^n\end{array}\right)$.

\begin{theorem} Let $A'=(A'_2\vert A'_3\vert...\vert A'_n)\in Mat((n-1)\times (n-1)^2,\mathbb{F})$ be a simple algebra such that
	the rank of the block-column matrix $A'_{bl}$ with blocks $A'_i+tr(A'_i)I_{n-1}$, where $i=2,3,...,n$, is $n-1$. 
	
	In this case an algebra $A=(A_1\vert A_2\vert ...\vert A_n)\in Mat(n\times n^2,\mathbb{F})$ is also a simple algebra, where $A_1=(a^i_{1j})\in Mat(n\times n,\mathbb{F})$, the row $(a^1_{12},a^1_{13},...,a^1_{1n})$ is not zero, 
	$A_i=\begin{pmatrix}a^1_{i1}&0\\ \overline{a}_{i1}&A'_i\end{pmatrix}$, $a^1_{i1}=-tr(A'_i)$ and $\overline{a}_{i1}\in \mathbb{F}^{n-1}$ are any column vectors such that the rank of the block-column matrix $A''_{bl}$ with blocks $(\overline{a}_{i1}\vert A'_i+tr(A'_i)I_{n-1})$, $i\in \{2,3,...,n\}$, is $n$. \end{theorem}	\begin{proof} Note that in this case $A^o_i=\begin{pmatrix}a^1_{1i}&0\\ \overline{a}_{1i}&(A')^o_i\end{pmatrix}$ whenever $i=2,3,...,n$.  Let a subspace $V$ of $\mathbb{F}^{n}$ be $\overline{A}$-invariant. Consider the projection map $p: V\rightarrow \mathbb{F}^{n-1}$, where $p(v)=p(v_1,v_2,...,v_n)=v'=(v_2,...,v_n)$ by definition.  
	
	If the first entry $v_1$ of $v\in  V$ is zero for all elements of $V$ then $V'=p(V)$ is a $\overline{A'}$-invariant subspace of $\mathbb{F}^{n-1}$. Indeed, if $i>1$ then $A_iv=\begin{pmatrix}0\\ A'_iv'\end{pmatrix}$, that is $A'_iv'\in V'$. Similarly, one has $(A')^o_iv'\in V'$. In this case $V'=\{0\}$ or $V'=\mathbb{F}^{n-1}$ due to the simplicity of $A'$. If $V'=\{0\}$ then $V=\{0\}$, if $V'=\mathbb{F}^{n-1}$ then there exists $v'\in \mathbb{F}^{n-1}$ such that the first entry of $A_1\begin{pmatrix}0\\ v'\end{pmatrix}$-is not zero as far as $(a^1_{12},a^1_{13},...,a^1_{1n})\neq 0$. This contradiction shows that $V=\{0\}$ is trivial subspace.
	
	If $i,j>1$ and there exists $v=\begin{pmatrix}1\\ v'\end{pmatrix}\in V$ then
	$A_iv=\begin{pmatrix}a^1_{i1}\\ \overline{a}_{i1}+A'_iv'\end{pmatrix}\in V$,\\ 
	$\begin{pmatrix}0\\ \overline{a}_{i1}+A'_iv'-a^1_{i1}v'\end{pmatrix}\in V$, 
	$A_j\begin{pmatrix}0\\ \overline{a}_{i1}+A'_iv'-a^1_{i1}v'\end{pmatrix}=\begin{pmatrix}0\\ A'_j(\overline{a}_{i1}+A'_iv'-a^1_{i1}v')\end{pmatrix}\in V$ and\\ $A^o_j\begin{pmatrix}0\\ \overline{a}_{i1}+A'_iv'-a^1_{i1}v'\end{pmatrix}=\begin{pmatrix}0\\ (A')^o_j(\overline{a}_{i1}+A'_iv'-a^1_{i1}v')\end{pmatrix}\in V.$
	
	It means that for every fixed $i>1$ one has  $\overline{A'}$-invariant subspace $V_i''$ generated by $\overline{a}_{i1}+(A'_i-a^1_{i1}I_{n-1})v'$, so $V_i''$ should be trivial. If $V_i''=\mathbb{F}^{n-1}$ for some $\begin{pmatrix}1\\ v'\end{pmatrix}\in V$ and $i>1$ then evidently $V=\mathbb{F}^{n}$-a trivial subspace. Let $V_i''=0$ be true for all $\begin{pmatrix}1\\ v'\end{pmatrix}\in V$ and $i>1$, that is $(A'_i-a^1_{i1}I_{n-1})v'=-\overline{a}_{i1}$.
	But $rk(A'_{bl})=n-1$ and $rk(A''_{bl})=n$, therefore
	there is no $v'$ for which $(A'_i-a^1_{i1}I_{n-1})v'=-\overline{a}_{i1}$ for all $i=2,3,...,n$. \end{proof}

\section{Some applications}
There is a result in
\cite{popov2016} (Theorem 4) which says that in the case of algebraically closed $\mathbb{F}$, $Char(\mathbb{F})\neq 2$ 
there is a nonempty open subset of $\mathcal{A}$  consisting of simple algebras which have only trivial automorphism. Theorem 3.1 shows that one can strengthen 
this result by changing "consisting of simple algebras with only trivial automorphism" to "consisting of simple algebras which have only trivial automorphism and trivial derivation". 

Let $\mathcal{A}^*_{2}$ stand for the set of all two-dimensional simple algebras which have only trivial derivation and trivial automorphism, for example, algebras given by MSC $$A=\left(
\begin{array}{cccc}
	\alpha_1 & \alpha_2 &1+\alpha_2 & \alpha_4 \\
	\beta_1 & -\alpha_1 & 1-\alpha_1 & -\alpha_2
\end{array}\right),\ \mbox{where}\ \mathbf{c}=(\alpha_1, \alpha_2, \alpha_4, \beta_1) \in \mathbb{F}^4,$$  $\beta_1\neq 2\alpha_1+\alpha_2$ or $\alpha_4\neq -\alpha_1-2\alpha_2,$ have these properties. If one applies construction of Theorem 5.1 to any element of this set to get a three-dimensional simple algebra then the obtained algebra also will have only trivial derivation and trivial automorphism due to the Theorem 3.2 and Theorem 4.1. Let $\mathcal{A}^*_{3}$ stand for the set of all such constructed algebras and etcetera, one can construct $\mathcal{A}^*_{n}$ for any $n>1$. The following result is valid.   \begin{proposition}  In the case of algebraically closed $\mathbb{F}$ the set $$\{g^{-1}Ag^{\otimes 2}: A\in\mathcal{A}^*_{n}, g\in GL(n,\mathbb{F})\}$$ is a dense subset of $\mathcal{A}$.
\end{proposition} 
Note that this result is true in $Char(\mathbb{F})= 2$  case as well.

Here is an application of the main results to the theory of $k$-local derivations(automorphisms) of finite dimensional algebras, see \cite{ayupov2016,jonson2001}. 	
Let $k\geq 1$ be any natural number.
\begin{definition} A map(not necessarily, linear) $B:\mathbb{A}\rightarrow \mathbb{A}$ is said to be a  $k$-local automorphism($k$-local derivation) if for any $x_1,x_2,...,x_k\in \mathbb{A}$ there exists $A_{(x_1,x_2,...,x_k)}\in Aut(\mathbb{A})$(respectively,  $A_{(x_1,x_2,...,x_k)}\in Der(\mathbb{A})$) such that $B(x_i)=A_{(x_1,x_2,...,x_k)}(x_i)$ for all  $k=1,2,...,k$.\end{definition}

Further $k-Locder(\mathbb{A})$ $(k-LocAut(\mathbb{A}))$ stands for the set of all k-local derivations(respectively, k-local automorphisms) of $\mathbb{A}$. The description of $k-Locder(\mathbb{A})$ $(k-LocAut(\mathbb{A}))$, for a given algebra  $\mathbb{A}$, is considered by many researchers. The following corollary states for the "majority" of algebras $\mathbb{A}$ these sets coincide with $Der(\mathbb{A})$(respectively, $Aut(\mathbb{A})$),  

\begin{corollary} If $\mathbb{F}$ is algebraically closed, $k\geq 1$ is any natural number then the set $$\{\mathbb{A}\in\mathcal{A}:\  k-LocDer(\mathbb{A})=Der(\mathbb{A}), k-LocAut(\mathbb{A})=Aut(\mathbb{A}) \}$$ is a dense subset of the variety of $n$-dimensional algebras.\end{corollary}

\section{Appendix section: Two-dimensional algebras}

In the case of two-dimensional algebras, given over any field, a complete classification, up to isomorphism, of all algebras which have only trivial derivation(automorphism) and all simple algebras can be provided. For the sake of completeness we would like to represent here the corresponding results without proofs as far as the proofs are not difficult.

In \cite{eshmirzaev24} one can find description of all derivations and automorphisms of representatives of isomorphic classes of two-dimensional algebras. One can use it to get the following results, in which we follow the notations of \cite{bekbaev24,eshmirzaev24}.

\begin{theorem} In $Char.(\mathbb{F})\neq 2,3$ case any two-dimensional algebra, which has only trivial derivation, is isomorphic to only one of the following such algebras given by MSC representatives	\begin{itemize}
		\item	$ A_{1}(\mathbf{c})=\left(
		\begin{array}{cccc}
			\alpha_1 & \alpha_2 &1+\alpha_2 & \alpha_4 \\
			\beta_1 & -\alpha_1 & 1-\alpha_1 & -\alpha_2
		\end{array}\right), \mbox{where}\ \mathbf{c}=(\alpha_1, \alpha_2, \alpha_4, \beta_1) \in \mathbb{F}^4,$ 
		\item	$ A_{2}(\mathbf{c})=\left(
		\begin{array}{cccc}
			\alpha_1 & 0 & 0 & \alpha_4 \\
			1& \beta_2 & 1-\alpha_1&0
		\end{array}\right), \mbox{where}\ \mathbf{c}=(\alpha_1,\alpha_4, \beta_2)\in \mathbb{F}^3, \alpha_4\neq 0,$		\item	$ A_{3}(\mathbf{c})=\left(
		\begin{array}{cccc}
			\alpha_1 & 0 & 0 & \alpha_4 \\
			0& \beta_2 & 1-\alpha_1&0
		\end{array}\right),$ where $\mathbf{c}=(\alpha_1,\alpha_4, \beta_2)\in \mathbb{F}^3$, $\alpha_4\neq 0$, 
		\item	$ A_{4}(\mathbf{c})=\left(
		\begin{array}{cccc}
			0 & 1 & 1 & 0 \\
			\beta _1& \beta_2 & 1&-1
		\end{array}\right),\ \mbox{where}\ \mathbf{c}=(\beta_1, \beta_2)\in \mathbb{F}^2,$ 
		\item	$ A_{6}(\mathbf{c})=\left(
		\begin{array}{cccc}
			\alpha_1 & 0 & 0 & \alpha_4 \\
			1& 1-\alpha_1 & -\alpha_1&0
		\end{array}\right),\ \mbox{where}\ \mathbf{c}=(\alpha_1,\alpha_4)\in \mathbb{F}^2,\ \alpha_4\neq 0,$
		\item	$ A_{7}(\mathbf{c})=\left(
		\begin{array}{cccc}
			\alpha_1 & 0 & 0 & \alpha_4 \\
			0&1-\alpha_1 & -\alpha_1&0
		\end{array}\right),$ where $\mathbf{c}=(\alpha_1,\alpha_4)\in \mathbb{F}^2$,\ $\alpha_4 \neq 0$,
		\item	$ A_{8}(\mathbf{c})=\left(
		\begin{array}{cccc}
			0 & 1 & 1 & 0 \\
			\beta _1& 1 & 0&-1
		\end{array}\right),\ \mbox{where}\ \mathbf{c}=\beta_1\in\mathbb{F},$\\ 
		\item	$A_{10}(\mathbf{c})=\left(
		\begin{array}{cccc}
			0 &1 & 1 &1 \\
			\beta_1 &0 &0 &-1
		\end{array}
		\right),$ where polynomial\\ $(\beta _1t^3-3t-1)(\beta_1t^2+\beta_1t+1)(\beta_1^2t^3+6\beta_1t^2+3\beta_1t+\beta_1-2)$ has no root in $\mathbb{F}$,
		\item	$A_{11}(\mathbf{c})=\left(
		\begin{array}{cccc}
			0 &0 & 0 &1 \\
			\beta_1 &0 &0 &0
		\end{array}
		\right),$ where polynomial $\beta _1 -t^3$ has no root in $\mathbb{F}$, $\mathbf{c}=\beta_1$,
		\item	$A_{12}(\mathbf{c})=\left(
		\begin{array}{cccc}
			0 & 1 & 1 &0  \\
			\beta_1 &0& 0 &-1
		\end{array}
		\right),\ \mbox{where}\ \mathbf{c}=\beta_1\in \mathbb{F},$ $\beta_1 \neq 0.$\
		
	\end{itemize}
\end{theorem}
\begin{theorem} In $Char.(\mathbb{F})=2$ case any two-dimensional algebra, which has only trivial derivation, is isomorphic to only one of the following such algebras given by MSC representatives	
	\begin{itemize}
		\item	$ A_{1,2}(\mathbf{c})=\left(
		\begin{array}{cccc}
			\alpha_1 & \alpha_2 &\alpha_2+1 & \alpha_4 \\
			\beta_1 & \alpha_1 & 1+\alpha_1 & \alpha_2
		\end{array}\right),\ \mbox{where}\ \mathbf{c}=(\alpha_1, \alpha_2, \alpha_4, \beta_1) \in \mathbb{F}^4,$
		\item	$ A_{2,2}(\mathbf{c})=\left(
		\begin{array}{cccc}
			\alpha_1 & 0 & 0 & \alpha_4 \\
			1& \beta_2 & 1+\alpha_1&0
		\end{array}\right),\ \mbox{where}\ \mathbf{c}=(\alpha_1,\alpha_4, \beta_2)\in \mathbb{F}^3,\\ \alpha_4\neq 0,$  or  $\beta_2\neq 1$,
		\item	$ A_{4,2}(\mathbf{c})=\left(
		\begin{array}{cccc}
			\alpha_1 & 1 & 1 & 0 \\
			\beta _1& \beta_2 & 1+\alpha_1&1
		\end{array}\right),\ \mbox{where}\ \mathbf{c}=(\alpha_1,\beta_1, \beta_2)\in \mathbb{F}^3,$ $\beta_2\neq 1,$ 
		\item	$ A_{5,2}(\mathbf{c})=\left(
		\begin{array}{cccc}
			\alpha_1 & 0 & 0 & \alpha_4 \\
			1&1+\alpha_1 & \alpha_1&0
		\end{array}\right),\ \mbox{where}\ \mathbf{c}=(\alpha_1,\alpha_4)\in \mathbb{F}^2,\ \alpha_4\neq 0,\\
		\ A_{5,2}(1,0)=\left(
		\begin{array}{cccc}
			1 & 0 & 0 & 0 \\
			1&0 & 1&0
		\end{array}\right),$
		\item	$ A_{6,2}(\mathbf{c})=\left(
		\begin{array}{cccc}
			\alpha_1 & 0 & 0 & \alpha_4 \\
			0&1+\alpha_1 & \alpha_1&0
		\end{array}\right),$ where $\mathbf{c}=(\alpha_1,\alpha_4)\in \mathbb{F}^2$,  $\alpha_1 \neq 1$ or $\alpha_4 \neq 0$,
		\item	$ A_{7,2}(\mathbf{c})=\left(
		\begin{array}{cccc}
			\alpha_1 & 1 & 1 & 0 \\
			\beta _1& 1+\alpha_1 & \alpha_1&1
		\end{array}\right),\ \mbox{where}\ \mathbf{c}=(\alpha_1,\beta_1)\in\mathbb{F}^2,$ $\alpha_1\neq 1, $
		\item	$A_{8,2}(\mathbf{c})=\left(
		\begin{array}{cccc}
			0 &1 & 1 &1 \\
			\beta_1 &0 &0 &1
		\end{array}
		\right),$ where polynomial\\ $(\beta _1t^3+t+1)(\beta _1t^2+\beta_1t+1)$ has no root in $\mathbb{F}$,
		\item $A_{9,2}(\mathbf{c})=\left(
		\begin{array}{cccc}
			0 &0 & 0 &1\\
			\beta_1 &0 &0 &0
		\end{array}
		\right),\  \mbox{where}\ \mathbf{c}=\beta_1\in\mathbb{F},$ polynomial $\beta_1+t^3$ has no root in $\mathbb{F}$,
		\item	$A_{10,2}(\mathbf{c})=\left(
		\begin{array}{cccc}
			1 & 1 & 1 & 0 \\
			\beta_1 &1& 1 &1
		\end{array}
		\right),\ \mbox{where}\ \mathbf{c}=\beta_1\in \mathbb{F}.$\ 	
	\end{itemize}	
\end{theorem}

\begin{theorem}  In $Char.(\mathbb{F})=3$ case any two-dimensional algebra, which has only trivial derivation, is isomorphic to only one of the following such algebras given by MSC representatives
	\begin{itemize}
		\item
		$ A_{1,3}(\mathbf{c})=\left(
		\begin{array}{cccc}
			\alpha_1 & \alpha_2 &\alpha_2+1 & \alpha_4 \\
			\beta_1 & -\alpha_1 & 1-\alpha_1 & -\alpha_2
		\end{array}\right), \mbox{where}\ \mathbf{c}=\alpha_1,\alpha_2,\alpha_4,\beta_1)\in \mathbb{F}^4,$
		\item $ A_{2,3}(\mathbf{c})=\left(
		\begin{array}{cccc}
			\alpha_1 & 0 & 0 & \alpha_4 \\
			1& \beta_2 & 1-\alpha_1&0
		\end{array}\right), \mbox{where}\ \mathbf{c}=(\alpha_1,\alpha_4, \beta_2)\in \mathbb{F}^3,\\ \alpha_4\neq 0,$ 
		\item	$ A_{3,3}(\mathbf{c})=\left(
		\begin{array}{cccc}
			\alpha_1 & 0 & 0 & \alpha_4 \\
			0& \beta_2 & 1-\alpha_1&0
		\end{array}\right),$\ \ where $\mathbf{c}=(\alpha_1,\alpha_4, \beta_2)\in \mathbb{F}^3$,\\ $\alpha_4 \neq 0$, 
		\item	$ A_{4,3}(\mathbf{c})=\left(
		\begin{array}{cccc}
			0 & 1 & 1 & 0 \\
			\beta _1& \beta_2 & 1&-1
		\end{array}\right),\ \mbox{where}\ \mathbf{c}=(\beta_1, \beta_2)\in \mathbb{F}^2,$\
		\item $A_{6,3}(\mathbf{c})=\left(
		\begin{array}{cccc}
			\alpha_1 & 0 & 0 & \alpha_4 \\
			1& 1-\alpha_1 & -\alpha_1&0
		\end{array}\right), \mbox{where}\ \mathbf{c}=(\alpha_1,\alpha_4)\in \mathbb{F}^2, \alpha_4\neq 0,$
		\item $ A_{7,3}(\mathbf{c})=\left(
		\begin{array}{cccc}
			\alpha_1 & 0 & 0 & \alpha_4 \\
			0&1-\alpha_1 & -\alpha_1&0
		\end{array}\right)$ where $\mathbf{c}=(\alpha_1,\alpha_4)\in \mathbb{F}^2$, $\alpha_4 \neq 0,$ 
		\item	$ A_{8,3}(\mathbf{c})=\left(
		\begin{array}{cccc}
			0 & 1 & 1 & 0 \\
			\beta _1& 1 & 0&-1
		\end{array}\right),\ \mbox{where}\ \mathbf{c}=\beta_1\in\mathbb{F},$
	\end{itemize}
\end{theorem}

\begin{theorem} In $Char.(\mathbb{F})\neq 2,3$ case any two-dimensional algebra, with only trivial automorphism, is isomorphic to only one of the following such algebras given by MSC representatives
	\begin{itemize}
		\item	$ A_{1}(\mathbf{c})=\left(
		\begin{array}{cccc}
			\alpha_1 & \alpha_2 &1+\alpha_2 & \alpha_4 \\
			\beta_1 & -\alpha_1 & 1-\alpha_1 & -\alpha_2
		\end{array}\right), \mbox{where}\ \mathbf{c}=(\alpha_1,\alpha_2,\alpha_4,\beta_1) \in \mathbb{F}^4,$ 
		\item	$ A_{2}(\mathbf{c})=\left(
		\begin{array}{cccc}
			\alpha_1 & 0 & 0 & \alpha_4 \\
			1& \beta_2 & 1-\alpha_1&0
		\end{array}\right),\ \mbox{where}\ \mathbf{c}=(\alpha_1,\alpha_4,\beta_2)\in\mathbb{F}^3, \alpha_4\neq 0,$
		\item	$ A_{4}(\mathbf{c})=\left(
		\begin{array}{cccc}
			0 & 1 & 1 & 0 \\
			\beta _1& \beta_2 & 1&-1
		\end{array}\right),\ \mbox{where}\ \mathbf{c}=(\beta_1, \beta_2)\in \mathbb{F}^2,$ 
		\item	$ A_{6}(\mathbf{c})=\left(
		\begin{array}{cccc}
			\alpha_1 & 0 & 0 & \alpha_4 \\
			1& 1-\alpha_1 & -\alpha_1&0
		\end{array}\right),\ \mbox{where}\ \mathbf{c}=(\alpha_1,\alpha_4)\in \mathbb{F}^2,\ \alpha_4\neq 0,$
		\item	$ A_{8}(\mathbf{c})=\left(
		\begin{array}{cccc}
			0 & 1 & 1 & 0 \\
			\beta _1& 1 & 0&-1
		\end{array}\right),\ \mbox{where}\ \mathbf{c}=\beta_1\in\mathbb{F},$
		\item	$A_{10}(\mathbf{c})=\left(
		\begin{array}{cccc}
			0 &1 & 1 &1 \\
			\beta_1 &0 &0 &-1
		\end{array}
		\right),$ where polynomial\\ $(\beta _1t^3-3t-1)(\beta_1t^2+\beta_1t+1)(\beta_1^2t^3+6\beta_1t^2+3\beta_1t+\beta_1-2)$ has no root in $\mathbb{F}$ and none of the following equations $\beta_1=\frac{(2d+1)^2}{d^2+d+1}$, $\beta_1=\frac{(2d+1)^2(d-1)}{d^3}$ has a solution in $\mathbb{F}$,
		\item	$A_{11}(\mathbf{c})=\left(
		\begin{array}{cccc}
			0 &0 & 0 &1 \\
			\beta_1 &0 &0 &0
		\end{array}
		\right),$ where polynomial $\beta _1 -t^3$ has no root in $\mathbb{F}$, $\mathbf{c}=\beta_1\neq 0$, and $d^3=1$ if and only if $d=1$.
	\end{itemize}
\end{theorem}

\begin{theorem} In $Char.(\mathbb{F})=2$ case any two-dimensional algebra, with only trivial automorphism, is isomorphic to only one of the following such algebras given by MSC representatives
	\begin{itemize}
		\item	$ A_{1,2}(\mathbf{c})=\left(
		\begin{array}{cccc}
			\alpha_1 & \alpha_2 &\alpha_2+1 & \alpha_4 \\
			\beta_1 & \alpha_1 & 1+\alpha_1 & \alpha_2
		\end{array}\right),\ \mbox{where}\ \mathbf{c}=(\alpha_1, \alpha_2, \alpha_4, \beta_1) \in \mathbb{F}^4,$
		\item	$A_{2,2}(\mathbf{c})=\left(
		\begin{array}{cccc}
			\alpha_1 & 0 & 0 & \alpha_4 \\
			1& \beta_2 & 1+\alpha_1&0
		\end{array}\right), \mbox{where}\ \mathbf{c}=(\alpha_1,\alpha_4,\beta_2)\in \mathbb{F}^3,\\ \alpha_4\neq 0,$  
		\item	$ A_{3,2}(\mathbf{c})=\left(
		\begin{array}{cccc}
			\alpha_1 & 0 & 0 & \alpha_4 \\
			0& \beta_2 & 1+\alpha_1&0
		\end{array}\right),$\ where $\mathbf{c}=(\alpha_1,\alpha_4,\beta_2)\in \mathbb{F}^3,\\ \alpha_4\neq 0,$\
		$ A_{3,2}(\alpha_1,0, 0)$, if $\mathbb{F}=\mathbb{Z}_2$,
		\item	$ A_{4,2}(\mathbf{c})=\left(
		\begin{array}{cccc}
			\alpha_1 & 1 & 1 & 0 \\
			\beta _1& \beta_2 & 1+\alpha_1&1
		\end{array}\right),\ \mbox{where}\ \mathbf{c}=(\alpha_1,\beta_1, \beta_2)\in \mathbb{F}^3, \beta_2=1$, 
		\item	$ A_{5,2}(\mathbf{c})=\left(
		\begin{array}{cccc}
			\alpha_1 & 0 & 0 & \alpha_4 \\
			1&1+\alpha_1 & \alpha_1&0
		\end{array}\right),\ \mbox{where}\ \mathbf{c}=(\alpha_1,\alpha_4)\in \mathbb{F}^2,\ \alpha_4\neq 0,$
		\item	$ A_{6,2}(\mathbf{c})=\left(
		\begin{array}{cccc}
			\alpha_1 & 0 & 0 & \alpha_4 \\
			0&1+\alpha_1 & \alpha_1&0
		\end{array}\right),$ where $\mathbf{c}=(\alpha_1,\alpha_4)\in \mathbb{F}^2$, $\alpha_4 \neq 0$,\\
		$ A_{6,2}(0,0)$, if $\mathbb{F}=\mathbb{Z}_2$.
		\item	$ A_{7,2}(\mathbf{c})=\left(
		\begin{array}{cccc}
			\alpha_1 & 1 & 1 & 0 \\
			\beta _1& 1+\alpha_1 & \alpha_1&1
		\end{array}\right),\ \mbox{where}\ \mathbf{c}=(\alpha_1,\beta_1)\in\mathbb{F}^2,$ $\alpha_1=1$,
		\item	$A_{8,2}(\mathbf{c})=\left(
		\begin{array}{cccc}
			0 &1 & 1 &1 \\
			\beta_1 &0 &0 &1
		\end{array}
		\right),$ where polynomial $(\beta _1t^3+t+1)(\beta _1t^2+\beta_1t+1)$ has no root in $\mathbb{F}$ and none of the following equations $\beta_1=\frac{1}{d^2+d+1}$, $\beta_1=\frac{d+1}{d^3}$ has a solution in $\mathbb{F}$, 
		\item $A_{9,2}(\mathbf{c})=\left(
		\begin{array}{cccc}
			0 &0 & 0 &1\\
			\beta_1 &0 &0 &0
		\end{array}
		\right),\  \mbox{where}\ \mathbf{c}=\beta_1\in\mathbb{F},$ polynomial $\beta_1+t^3$ has no root in $\mathbb{F}$ and $d^3=1$ if and only if $d=1$,
		\item	$A_{11,2}(0)=\left(
		\begin{array}{cccc}
			0 & 1 & 1 & 0 \\
			0 &0& 0 &1
		\end{array}
		\right)$, if $\mathbb{F}=\mathbb{F}_2$-two element field.	
		
	\end{itemize}	
\end{theorem}

\begin{theorem} In $Char.(\mathbb{F})=3$ case any two-dimensional algebra, with only trivial automorphism, is isomorphic to only one of the following such algebras given by MSC representatives
	\begin{itemize}
		\item
		$ A_{1,3}(\mathbf{c})=\left(
		\begin{array}{cccc}
			\alpha_1 & \alpha_2 &\alpha_2+1 & \alpha_4 \\
			\beta_1 & -\alpha_1 & 1-\alpha_1 & -\alpha_2
		\end{array}\right), \mbox{where}\ \mathbf{c}=(\alpha_1,\alpha_2,\alpha_4,\beta_1)\in \mathbb{F}^4,$
		\item	$ A_{2,3}(\mathbf{c})=\left(
		\begin{array}{cccc}
			\alpha_1 & 0 & 0 & \alpha_4 \\
			1& \beta_2 & 1-\alpha_1&0
		\end{array}\right),\ \mbox{where}\ \mathbf{c}=(\alpha_1,\alpha_4, \beta_2)\in \mathbb{F}^3,\\ \alpha_4\neq 0,$ 
		\item	$ A_{4,3}(\mathbf{c})=\left(
		\begin{array}{cccc}
			0 & 1 & 1 & 0 \\
			\beta _1& \beta_2 & 1&-1
		\end{array}\right),\ \mbox{where}\ \mathbf{c}=(\beta_1, \beta_2)\in \mathbb{F}^2,$
		\item	$ A_{6,3}(\mathbf{c})=\left(
		\begin{array}{cccc}
			\alpha_1 & 0 & 0 & \alpha_4 \\
			1& 1-\alpha_1 & -\alpha_1&0
		\end{array}\right),\ \mbox{where}\ \mathbf{c}=(\alpha_1,\alpha_4)\in \mathbb{F}^2,\\ \alpha_4\neq 0,$
		\item	$ A_{8,3}(\mathbf{c})=\left(
		\begin{array}{cccc}
			0 & 1 & 1 & 0 \\
			\beta _1& 1 & 0&-1
		\end{array}\right),\ \mbox{where}\ \mathbf{c}=\beta_1\in\mathbb{F},$ 
		\item	$ A_{9,3}(\mathbf{c})=\left(
		\begin{array}{cccc}
			0 &1 & 1 &1 \\
			\beta_1 &0 &0 &-1
		\end{array}
		\right),$ where polynomial $(\beta _1 -t^3)(\beta _1t^2+\beta _1t+1)(\beta_1^2t^3+\beta_1-2)$ has no root in $\mathbb{F}$, the following equations $\beta_1=\frac{(2d+1)^2}{d^2+d+1}$ has no solution in $\mathbb{F},$
		\item	$A_{10,3}(\mathbf{c})=\left(
		\begin{array}{cccc}
			0 &0 & 0 &1 \\
			\beta_1 &0 &0 &0
		\end{array}
		\right),$\ where polynomial $\beta_1 -t^3$ has no root, $\mathbf{c}=\beta_1\neq 0$ and $d^3=1$ if and only if $d=1$.
		
	\end{itemize}
\end{theorem}

In \cite{ahmed1}, in particular, one can find a classification, up to isomorphism, of two-dimensional simple algebras given over a field where every second and third order polynomial has a root. Here we would like to close finally this problem for two-dimensional algebras over any field. For that one can relay on a complete classification, up to isomorphism, of all two-dimensional algebras presented in \cite{bekbaev24} and Proposition 5.1, Proposition 6.1 presented in \cite{ahmed1}. 

\begin{theorem} In $Char.(\mathbb{F})\neq 2,3$ case any two-dimensional simple algebra is isomorphic to one of the following such algebras given by MSC representatives	\begin{itemize}
		\item	$ A_{1}(\mathbf{c})=\left(
		\begin{array}{cccc}
			\alpha_1 & \alpha_2 &1+\alpha_2 & \alpha_4 \\
			\beta_1 & -\alpha_1 & 1-\alpha_1 & -\alpha_2
		\end{array}\right),\ \mbox{where}\ \mathbf{c}=(\alpha_1, \alpha_2, \alpha_4, \beta_1) \in \mathbb{F}^4$;\\ $\beta_1\neq 2\alpha_1+\alpha_2$ or $\alpha_4\neq -\alpha_1-2\alpha_2$, 
		\item	$A_{2}(\mathbf{c})=\left(
		\begin{array}{cccc}
			\alpha_1 & 0 & 0 & \alpha_4 \\
			1& \beta_2 & 1-\alpha_1&0
		\end{array}\right), \mbox{where}\ \mathbf{c}=(\alpha_1,\alpha_4,\beta_2)\in \mathbb{F}^3,\ \alpha_4\neq 0;$\
		\item $ A_{3}(\mathbf{c})=\left(
		\begin{array}{cccc}
			\alpha_1 & 0 & 0 & \alpha_4 \\
			0& \beta_2 & 1-\alpha_1&0
		\end{array}\right),$ where $\mathbf{c}=(\alpha_1,\alpha_4, \beta_2)\in \mathbb{F}^3$;\\ $\beta_2 \neq 1-\alpha_1, \alpha_4\neq 0,$ or $\beta_2 =1-\alpha_1, \alpha_4\neq 0, \alpha_1\neq 1, \alpha_1\neq \frac{1}{2},$% \alpha_4=-\frac{1}{4}, \beta_2=1$ or $\alpha_1=\beta_2=\frac{1}{2}, \alpha_4\neq 0,$  or $\alpha_1=\beta_2=\frac{1}{2}, \alpha_4\neq 0,$ %\alpha_4=-\frac{1}{8}, \beta_2=\frac{1}{2}$, or\\ $\alpha_4\neq 0, \alpha_1\neq 1, \alpha_1\neq 1+4\alpha_4,$
		\item $ A_{4}(\mathbf{c})=\left(
		\begin{array}{cccc}
			0 & 1 & 1 & 0 \\
			\beta _1& \beta_2 & 1&-1
		\end{array}\right),\ \mbox{where}\ \mathbf{c}=(\beta_1, \beta_2)\in \mathbb{F}^2;$  $\beta _2\neq 1$, or $\beta _1\neq -\frac{1}{4}, \beta _2= 1$,
		\item $ A_{6}(\mathbf{c})=\left(
		\begin{array}{cccc}
			\alpha_1 & 0 & 0 & \alpha_4 \\
			1& 1-\alpha_1 & -\alpha_1&0
		\end{array}\right),\ \mbox{where}\ \mathbf{c}=(\alpha_1,\alpha_4)\in \mathbb{F}^2;\ \alpha_4\neq 0;$
		\item $ A_{7}(\mathbf{c})=\left(
		\begin{array}{cccc}
			\alpha_1 & 0 & 0 & \alpha_4 \\
			0&1-\alpha_1 & -\alpha_1&0
		\end{array}\right)$, where $\mathbf{c}=(\alpha_1,\alpha_4)\in \mathbb{F}^2$; $\alpha_4\neq 0$,\\ 
		\item $ A_{8}(\mathbf{c})=\left(
		\begin{array}{cccc}
			0 & 1 & 1 & 0 \\
			\beta _1& 1 & 0&-1
		\end{array}\right),\ \mbox{where}\ \mathbf{c}=\beta_1\in\mathbb{F};$ 
		\item $A_{10}(\mathbf{c})=\left(
		\begin{array}{cccc}
			0 &1 & 1 &1 \\
			\beta_1 &0 &0 &-1
		\end{array}
		\right),$ where polynomial $(\beta _1t^3-3t-1)(\beta_1t^2+\beta_1t+1)(\beta_1^2t^3+6\beta_1t^2+3\beta_1t+\beta_1-2)$ has no root in $\mathbb{F}$, 
		\item $A_{11}(\mathbf{c})=\left(
		\begin{array}{cccc}
			0 &0 & 0 &1 \\
			\beta_1 &0 &0 &0
		\end{array}
		\right),$ where polynomial $\beta _1 -t^3$ has no root in $\mathbb{F}$,  $\mathbf{c}=\beta_1\neq 0$,
		\item $A_{12}(\mathbf{c})=\left(
		\begin{array}{cccc}
			0 & 1 & 1 &0  \\
			\beta_1 &0& 0 &-1
		\end{array}
		\right),\ \mbox{where}\ \mathbf{c}=\beta_1\in \mathbb{F}; \beta_1\neq 0.
		$  \end{itemize}
	
\end{theorem}

\begin{theorem} In $Char.(\mathbb{F})= 2$ case any two-dimensional simple algebra is isomorphic to one of the following such algebras given by MSC representatives 	\begin{itemize}
		\item	
		$ A_{1,2}(\mathbf{c})=\left(
		\begin{array}{cccc}
			\alpha_1 & \alpha_2 &\alpha_2+1 & \alpha_4 \\
			\beta_1 & \alpha_1 & 1+\alpha_1 & \alpha_2
		\end{array}\right),\ \mbox{where}\ \mathbf{c}=(\alpha_1, \alpha_2, \alpha_4, \beta_1) \in \mathbb{F}^4;$\\ $\alpha_4\neq \alpha_1$ or $\beta_1\neq \alpha_2$,
		\item $ A_{2,2}(\mathbf{c})=\left(
		\begin{array}{cccc}
			\alpha_1 & 0 & 0 & \alpha_4 \\
			1& \beta_2 & 1+\alpha_1&0
		\end{array}\right), \mbox{where}\ \mathbf{c}=(\alpha_1,\alpha_4,\beta_2)\in \mathbb{F}^3,\ \alpha_4\neq 0;$\ $\beta_2\neq 1+\alpha_1$ or $\beta_2= 1+\alpha_1=\alpha_4\neq 0$,
		\item 	$ A_{3,2}(\mathbf{c})=\left(
		\begin{array}{cccc} \alpha_1 & 0 & 0 & \alpha_4 \\
			0& \beta_2 & 1+\alpha_1&0
		\end{array}\right),$ where $\mathbf{c}=(\alpha_1,\alpha_4, \beta_2)\in \mathbb{F}^3$;\\ $\beta_2\neq 1+\alpha_1, \alpha_4\neq 0$ or $\beta_2=0, \alpha_1=1, \alpha_4\neq 0$,
		\item	$ A_{4,2}(\mathbf{c})=\left(
		\begin{array}{cccc}
			\alpha_1 & 1 & 1 & 0 \\
			\beta _1& \beta_2 & 1+\alpha_1&1
		\end{array}\right),\ \mbox{where}\ \mathbf{c}=(\alpha_1,\beta_1, \beta_2)\in \mathbb{F}^3;$\\ $\beta_2\neq 1+\alpha_1,$ or $\beta_2=0, \alpha_1=1, \beta_1=a+a^2$ for some $a\in \mathbb{F}$,
		\item	$ A_{5,2}(\mathbf{c})=\left(
		\begin{array}{cccc}
			\alpha_1 & 0 & 0 & \alpha_4 \\
			1&1+\alpha_1 & \alpha_1&0
		\end{array}\right),\ \mbox{where}\ \mathbf{c}=(\alpha_1,\alpha_4)\in \mathbb{F}^2,\ \alpha_4\neq 0;$
		\item	$ A_{6,2}(\mathbf{c})=\left(
		\begin{array}{cccc}
			\alpha_1 & 0 & 0 & \alpha_4 \\
			0&1+\alpha_1 & \alpha_1&0
		\end{array}\right),$ where $\mathbf{c}=(\alpha_1,\alpha_4)\in \mathbb{F}^2$, $ \alpha_4\neq 0$;\ 
		\item	$A_{7,2}(\mathbf{c})=\left(
		\begin{array}{cccc}
			\alpha_1 & 1 & 1 & 0 \\
			\beta _1& 1+\alpha_1 & \alpha_1&1
		\end{array}\right),\ \mbox{where}\ \mathbf{c}=(\alpha_1,\beta_1)\in\mathbb{F}^2;$
		\item	$A_{8,2}(\mathbf{c})=\left(
		\begin{array}{cccc}
			0 &1 & 1 &1 \\
			\beta_1 &0 &0 &1
		\end{array}
		\right),$ where polynomial\\ $(\beta _1t^3+t+1)(\beta _1t^2+\beta_1t+1)$\ has no root in $\mathbb{F}$;
		\item	$A_{9,2}(\mathbf{c})=\left(
		\begin{array}{cccc}
			0 &0 & 0 &1\\
			\beta_1 &0 &0 &0
		\end{array}
		\right),\  \mbox{where}\ \mathbf{c}=\beta_1\in\mathbb{F},$ polynomial $\beta_1+t^3$ has no root in $\mathbb{F}$;
		\item $A_{11,2}=\left(
		\begin{array}{cccc}
			0 & 1 & 1 & 0 \\
			\beta_1 &0& 0 &1
		\end{array}
		\right),$ where $\beta_1=a^2$ for some $a\in \mathbb{F}$.
	\end{itemize}
\end{theorem}

\begin{theorem} In $Char.(\mathbb{F})=3$ case any two-dimensional simple algebra is isomorphic to one of the following such algebras given by MSC representatives 	\begin{itemize}
		\item	
		$ A_{1,3}(\mathbf{c})=\left(
		\begin{array}{cccc}
			\alpha_1 & \alpha_2 &\alpha_2+1 & \alpha_4 \\
			\beta_1 & -\alpha_1 & 1-\alpha_1 & -\alpha_2
		\end{array}\right),\ \mbox{where}\ \mathbf{c}=(\alpha_1, \alpha_2, \alpha_4, \beta_1) \in \mathbb{F}^4;$\ $\beta_1\neq 2\alpha_1+\alpha_2$ or $\alpha_4\neq -\alpha_1-2\alpha_2$, 
		\item $ A_{2,3}(\mathbf{c})=\left(
		\begin{array}{cccc}
			\alpha_1 & 0 & 0 & \alpha_4 \\
			1& \beta_2 & 1-\alpha_1&0
		\end{array}\right),\ \mbox{where}\ \mathbf{c}=(\alpha_1,\alpha_4,\beta_2)\in \mathbb{F}^3,\\ \alpha_4\neq 0,$
		\item $ A_{3,3}(\mathbf{c})=\left(
		\begin{array}{cccc}
			\alpha_1 & 0 & 0 & \alpha_4 \\
			0& \beta_2 & 1-\alpha_1&0
		\end{array}\right),$\ \ where $\mathbf{c}=(\alpha_1,\alpha_4, \beta_2)\in \mathbb{F}^3$;\\ $\beta_2 \neq 1-\alpha_1, \alpha_4\neq 0,$ or $\beta_2 =1-\alpha_1, \alpha_4\neq 0, \alpha_1\neq 1, \alpha_1\neq \frac{1}{2},$
		\item $ A_{4,3}(\mathbf{c})=\left(
		\begin{array}{cccc}
			0 & 1 & 1 & 0 \\
			\beta _1& \beta_2 & 1&-1
		\end{array}\right),\ \mbox{where}\ \mathbf{c}=(\beta_1, \beta_2)\in \mathbb{F}^2;$ $\beta _2\neq 1$, or $\beta _1\neq -1, \beta _2= 1$,
		\item $ A_{6,3}(\mathbf{c})=\left(
		\begin{array}{cccc}
			\alpha_1 & 0 & 0 & \alpha_4 \\
			1& 1-\alpha_1 & -\alpha_1&0
		\end{array}\right),\ \mbox{where}\ \mathbf{c}=(\alpha_1,\alpha_4)\in\mathbb{F}^2, \alpha_4\neq 0,$
		\item $ A_{7,3}(\mathbf{c})=\left(
		\begin{array}{cccc}
			\alpha_1 & 0 & 0 & \alpha_4 \\
			0&1-\alpha_1 & -\alpha_1&0
		\end{array}\right),$ where $\mathbf{c}=(\alpha_1,\alpha_4)\in \mathbb{F}^2$, \ \
		\item $ A_{8,3}(\mathbf{c})=\left(
		\begin{array}{cccc}
			0 & 1 & 1 & 0 \\
			\beta _1& 1 & 0&-1
		\end{array}\right),\ \mbox{where}\ \mathbf{c}=\beta_1\in\mathbb{F},$\\
		\item $ A_{9,3}(\mathbf{c})=\left(
		\begin{array}{cccc}
			0 &1 & 1 &1 \\
			\beta_1 &0 &0 &-1
		\end{array}
		\right),$ where polynomial\\ $(\beta _1 -t^3)(\beta _1t^2+\beta _1t+1)(\beta_1^2t^3+\beta_1-2)$ has no root in $\mathbb{F}$,
		\item $A_{10,3}(\mathbf{c})=\left(
		\begin{array}{cccc}
			0 &0 & 0 &1 \\
			\beta_1 &0 &0 &0
		\end{array}
		\right),$\ where polynomial $\beta_1 -t^3$ has no root, $\mathbf{c}=\beta_1\neq 0$,
		\item $A_{11,3}(\mathbf{c})=\left(
		\begin{array}{cccc}
			0 & 1 & 1 &0  \\
			\beta_1 &0& 0 &-1
		\end{array}
		\right),\ \mbox{where}\ \mathbf{c}=\beta_1\in \mathbb{F},\ \beta_1\neq 0,$
		\item $A_{12,3}=\left(
		\begin{array}{cccc}
			1 & 0 & 0 & 0 \\
			1 &-1&-1 &0\end{array}\right).$\
	\end{itemize}
\end{theorem}

One can use the above presented results to describe elements of $\mathcal{A}^*_{2}$
up to isomorphism. For example, it is here in $Char.(\mathbb{F})\neq 2,3$ case. 
\begin{corollary} In $Char.(\mathbb{F})\neq 2,3$ case any two-dimensional algebra, which is simple, has only trivial derivation and trivial automorphism, is isomorphic to only one of the following such algebras given by MSC representatives	\begin{itemize}
		\item	$ A_{1}(\mathbf{c})=\left(
		\begin{array}{cccc}
			\alpha_1 & \alpha_2 &1+\alpha_2 & \alpha_4 \\
			\beta_1 & -\alpha_1 & 1-\alpha_1 & -\alpha_2
		\end{array}\right),\ \mbox{where}\ \mathbf{c}=(\alpha_1, \alpha_2, \alpha_4, \beta_1) \in \mathbb{F}^4,$ $\beta_1\neq 2\alpha_1+\alpha_2$ or $\alpha_4\neq -\alpha_1-2\alpha_2$,
		\item	$ A_{2}(\mathbf{c})=\left(
		\begin{array}{cccc}
			\alpha_1 & 0 & 0 & \alpha_4 \\
			1& \beta_2 & 1-\alpha_1&0
		\end{array}\right),\ \mbox{where}\ \mathbf{c}=(\alpha_1,\alpha_4,\beta_2)\in \mathbb{F}^3, \alpha_4\neq 0,$	
		\item $ A_{4}(\mathbf{c})=\left(
		\begin{array}{cccc}
			0 & 1 & 1 & 0 \\
			\beta _1& \beta_2 & 1&-1
		\end{array}\right),\ \mbox{where}\ \mathbf{c}=(\beta_1, \beta_2)\in \mathbb{F}^2;$  $\beta _2\neq 1$, or $\beta _1\neq -\frac{1}{4}, \beta _2= 1$,
		\item $ A_{6}(\mathbf{c})=\left(
		\begin{array}{cccc}
			\alpha_1 & 0 & 0 & \alpha_4 \\
			1& 1-\alpha_1 & -\alpha_1&0
		\end{array}\right),\ \mbox{where}\ \mathbf{c}=(\alpha_1,\alpha_4)\in \mathbb{F}^2;\ \alpha_4\neq 0;$
		\item $ A_{8}(\mathbf{c})=\left(
		\begin{array}{cccc}
			0 & 1 & 1 & 0 \\
			\beta _1& 1 & 0&-1
		\end{array}\right),\ \mbox{where}\ \mathbf{c}=\beta_1\in\mathbb{F};$ 
		\item $A_{10}(\mathbf{c})=\left(
		\begin{array}{cccc}
			0 &1 & 1 &1 \\
			\beta_1 &0 &0 &-1
		\end{array}
		\right),$ where polynomial\\ $(\beta _1t^3-3t-1)(\beta_1t^2+\beta_1t+1)(\beta_1^2t^3+6\beta_1t^2+3\beta_1t+\beta_1-2)$ has no root in $\mathbb{F}$ and none of the following equations $\beta_1=\frac{(2d+1)^2}{d^2+d+1}$, $\beta_1=\frac{(2d+1)^2(d-1)}{d^3}$ has a solution in $\mathbb{F}$, 
		\item $A_{11}(\mathbf{c})=\left(
		\begin{array}{cccc}
			0 &0 & 0 &1 \\
			\beta_1 &0 &0 &0
		\end{array}
		\right),$ where polynomial $\beta _1 -t^3$ has no root in $\mathbb{F}$,  $\mathbf{c}=\beta_1\neq 0$ and $d^3=1$ if and only if $d=1$.
	\end{itemize}
\end{corollary}

One of the interesting questions is the description the elements, up to isomorphism, of the sets  $\mathcal{A}_{n,triv.der}\setminus \mathcal{A}_{n,triv.aut}$ and $\mathcal{A}_{n,triv.aut}\setminus\mathcal{A}_{n,triv.der}$.	

Relying on the classification results of this section one can answer the above question, in $n=2$ case, in the following way. 

\begin{theorem} In $Char.(\mathbb{F})\neq 2,3$ case the inclusion			$$\mathcal{A}_{2,triv.aut}\subset\mathcal{A}_{2,triv.der}$$ is valid		and the difference $\mathcal{A}_{2,triv.der}\setminus\mathcal{A}_{2,triv.aut}$ consists of 
	algebras isomorphic to one of the following algebras: 
	\begin{itemize}
		\item	$ A_{3}(\mathbf{c})=\left(
		\begin{array}{cccc}
			\alpha_1 & 0 & 0 & \alpha_4 \\
			0& \beta_2 & 1-\alpha_1&0
		\end{array}\right),$ where $\mathbf{c}=(\alpha_1,\alpha_4, \beta_2)\in \mathbb{F}^3$, $\alpha_4\neq 0$, 
		\item	$ A_{7}(\mathbf{c})=\left(
		\begin{array}{cccc}
			\alpha_1 & 0 & 0 & \alpha_4 \\
			0&1-\alpha_1 & -\alpha_1&0
		\end{array}\right),$ where $\mathbf{c}=(\alpha_1,\alpha_4)\in \mathbb{F}^2$,\ $\alpha_4 \neq 0$,
		\item	$A_{10}(\mathbf{c})=\left(
		\begin{array}{cccc}
			0 &1 & 1 &1 \\
			\beta_1 &0 &0 &-1
		\end{array}
		\right),$ where polynomial\\ $(\beta _1t^3-3t-1)(\beta_1t^2+\beta_1t+1)(\beta_1^2t^3+6\beta_1t^2+3\beta_1t+\beta_1-2)$ has no root in $\mathbb{F}$ and at least one of the following equations $\beta_1=\frac{(2d+1)^2}{d^2+d+1}$, $\beta_1=\frac{(2d+1)^2(d-1)}{d^3}$ has a solution in $\mathbb{F}$,
		\item	$A_{11}(\mathbf{c})=\left(
		\begin{array}{cccc}
			0 &0 & 0 &1 \\
			\beta_1 &0 &0 &0
		\end{array}
		\right),$ where polynomial $\beta _1 -t^3$ has no root in $\mathbb{F}$, $\mathbf{c}=\beta_1\neq 0$, and there exists $d\neq 1$ for which $d^3=1$,
		\item	$A_{12}(\mathbf{c})=\left(
		\begin{array}{cccc}
			0 & 1 & 1 &0  \\
			\beta_1 &0& 0 &-1
		\end{array}
		\right),\ \mbox{where}\ \mathbf{c}=\beta_1\in \mathbb{F},$ $\beta_1 \neq 0.$\
	\end{itemize}
	In  
	$Char.(\mathbb{F})= 2$ case the difference $\mathcal{A}_{2,triv.der}\setminus\mathcal{A}_{2,triv.aut}$ consists of 
	algebras isomorphic to one of the following algebras: \begin{itemize}
		\item	$ A_{2,2}(\mathbf{c})=\left(
		\begin{array}{cccc}
			\alpha_1 & 0 & 0 & \alpha_4 \\
			1& \beta_2 & 1+\alpha_1&0
		\end{array}\right),\ \mbox{where}\ \mathbf{c}=(\alpha_1,\alpha_4, \beta_2)\in \mathbb{F}^3,\\ \alpha_4= 0,$ $\beta_2\neq 1$, 
		\item
		$A_{4,2}(\beta_1,\beta_2)$, where $\beta_1,\beta_2\in \mathbb{F}$, $\beta_2\neq 1$;
		\item $A_{5,2}(1,0)$;
		\item $A_{6,2}(\alpha_1,0)$, where $\alpha_1\in \mathbb{F}$, $\alpha_1\neq 1$;\
		\item $A_{7,2}(\alpha_1,\beta_1)$, where $\alpha_1,\beta_1\in \mathbb{F}$, $\alpha_1\neq 1$;\
		\item $A_{8,2}(\beta_1)$, where polynomial $(\beta _1t^3+t+1)(\beta _1t^2+\beta_1t+1)$ has no root in $\mathbb{F}$ and at least one of the equations $\beta_1=\frac{1}{d^2+d+1}$, $\beta_1=\frac{d+1}{d^3}$ has a solution in $\mathbb{F}$;\
		\item $A_{9,2}(\beta_1)$, where $\mathbf{c}=\beta_1\in\mathbb{F},$ polynomial $\beta_1+t^3$ has no root in $\mathbb{F}$ and there exists $d\neq 1$ for which $d^3=1$;\
		\item  $A_{10,2}(\beta_1)$, where $\mathbf{c}=\beta_1\in\mathbb{F}.$ \end{itemize} 
	In  
	$Char.(\mathbb{F})= 2$ case	the difference	 $\mathcal{A}_{2,triv.aut}\setminus\mathcal{A}_{2,triv.der}$ consists 
	algebra isomorphic to one of the following algebras: \begin{itemize}
		\item
		$A_{3,2}(\alpha_1,\alpha_4,\beta_2)$, where $\alpha_1,\alpha_4, \beta_2\in \mathbb{F}$, $\alpha_4\neq 0$;\\
		$A_{3,2}(\alpha_1,0,0)$, where $\alpha_1\in \mathbb{F}$,		
		if $\mathbb{F}=\mathbb{F}_2$;\ 
		\item $A_{4,2}(\alpha_1,\beta_1,1)$, where $\alpha_1,\beta_1\in \mathbb{F}$;\
		%$A_{6,2}(\alpha_1, 0)$,
		\item $A_{7,2}(1,\beta_1)$, where $\beta_1\in \mathbb{F}$,
		\item	$A_{11,2}(0)$, if $\mathbb{F}=\mathbb{F}_2$.	
	\end{itemize}
	In  
	$Char.(\mathbb{F})= 3$ case the difference $\mathcal{A}_{2,triv.der}\setminus\mathcal{A}_{2,triv.aut}$ consists of 
	algebras isomorphic to one of the following algebras: \begin{itemize}
		\item
		$A_{3,3}(\alpha_1,\alpha_4,\beta_2)$, where $\alpha_1,\alpha_4, \beta_2\in \mathbb{F}$, $\alpha_4\neq 0$;\
		\item $A_{7,3}(\alpha_1,\alpha_4)$, where $\alpha_1,\alpha_4\in \mathbb{F}$, $\alpha_4\neq 0$ 		 
	\end{itemize} In  
	$Char.(\mathbb{F})= 3$ case the difference $\mathcal{A}_{2,triv.aut}\setminus\mathcal{A}_{2,triv.der}$ consists 
	algebra isomorphic to one of the following algebras: \begin{itemize}
		\item $A_{9,3}(\beta_1)$, where polynomial $(\beta _1 -t^3)(\beta _1t^2+\beta _1t+1)(\beta_1^2t^3+\beta_1-2)$ has no root in $\mathbb{F}$ and the  equation $\beta_1=\frac{(2d+1)^2}{d^2+d+1}$ has has no solution in $\mathbb{F}$,
		\item	$A_{10,3}(\mathbf{c})=\left(
		\begin{array}{cccc}
			0 &0 & 0 &1 \\
			\beta_1 &0 &0 &0
		\end{array}
		\right),$\ where polynomial $\beta_1 -t^3$ has no root, $\mathbf{c}=\beta_1\neq 0$ and $d^3=1$ if and only if $d=1$.
	\end{itemize} 
\end{theorem}


\begin{thebibliography}{9}
	
	\bibitem{ahmed1} Ahmed H., Bekbaev U., Rakhimov I., Complete classification of two-dimensional
	algebras, \textit{AIP Conference Proceedings}, {\bf 1830} (2017), 1-12.
	
	\bibitem{ahmed2} Ahmed H., Bekbaev U., Rakhimov I. The automorphism groups and derivation algebras of two-dimensional algebras, \textit{Journal of Generalized Lie Theory and Applications}, 12-1 (2018), pp.1-9, DOI 10.4172/1736-4337.1000290.
	
	\bibitem{ahmed3}	Ahmed, H., Bekbaev, U., Rakhimov, I., Subalgebras, idempotents, left(right) ideals and left quasiunits of two-dimensional algebras, \textit {International Journal of Algebra and Computation}, Vol. 30, No. 5 (2020) 903–929,  World Scientific Publishing Company, DOI: 10.1142/S0218196720500253
	
	\bibitem{bekbaev23} Bekbaev U., Eshmirzaev Sh., On local, 2-local automorphisms and derivations of two-dimensional algebras,\textit{Abstracts of the scientific and practical conference "Actual problems of mathematical modeling and information technology"}, Nukus, May 2-3, 2023, V1, pp. 18-20
	
	\bibitem{bekbaev24} Bekbaev U., Classification of two-dimensional algebras over any basic field, 2023, \textit{AIP Conference Proceedings}, 2880, 030001, https://doi.org/10.1063/5.0165726
	
	\bibitem{eshmirzaev24} Eshmirzaev Sh., Bekbaev U., A description of automorphisms and derivations of all two-dimensional algebras over any basic field,\textit{arXiv}:2409.08814 math.RA
	
	\bibitem{ayupov2016} Ayupov Sh.A., Kudaybergenov K.K., Peralta A.M. A survey on local and 2-local derivations on C*-and von Neumann algebras, in Topics in Functional Analysis and Algebra,
	\textit{Contemporary Mathematics},-2016, vol. 672, Amer. Math. Soc., Providence, RI, pp. 73-126.
	
	\bibitem{jonson2001} Johnson B. Local derivations on C*-algebras are derivations, \textit{Trans. Amer. Math. Soc.} 353(2001), pp. 313-325.
	%\bibitem{LS} Larson D., Sourour A.
	%Local derivations and local automorphisms,
	%\textit{Proc. Sympos. Pure Math.}, 51 (1990), pp. %187-194.
	
	\bibitem{popov2016} 
	Popov Vladimir L., 
	Algebras of General Type: 
	Rational Parametrization and Normal Forms, \textit{Proceedings of the Steklov Institute of Mathematics}, 2016, Vol. 292, pp. 202–215, Pleiades Publishing, Ltd., 2016.
	
	
\end{thebibliography}
\end{document}